\documentclass[12 pt]{article}
\usepackage{fullpage}

\usepackage{lipsum}
\usepackage{amsfonts}
\usepackage{graphicx}
\usepackage{epstopdf}
\usepackage{standalone}
\usepackage{algcompatible}

\ifpdf
  \DeclareGraphicsExtensions{.eps,.pdf,.png,.jpg}
\else
  \DeclareGraphicsExtensions{.eps}
\fi

\usepackage{amsopn}

\usepackage{epsfig} 
\usepackage{amsmath} 
\usepackage{amssymb}  
\usepackage{bm}
\DeclareMathAlphabet{\mathcal}{OMS}{cmsy}{m}{n}
\usepackage[english]{babel}
\usepackage{url}
\usepackage{algorithm}
\usepackage{color}
\usepackage{pgfplots}
\usepackage{siunitx}
\usepackage{tikz}
\usepackage{tkz-euclide}
\usepackage{chngcntr}
\usepackage{verbatim}
\usepackage{enumerate}

\definecolor{ao(english)}{rgb}{0.0, 0.5, 0.0}
\usepackage[colorlinks, citecolor = {ao(english)}, linkcolor = {ao(english)}]{hyperref} 
\usepackage{cleveref}
\usepackage{aliascnt}
\usetikzlibrary{calc}
\usepackage{subcaption}
\usepackage{tikz}
\usepackage{pgfplots}
\usetikzlibrary{arrows,shapes,trees,calc,positioning,patterns,decorations.pathmorphing,decorations.markings}
\usetikzlibrary{matrix}
\usepgfplotslibrary{groupplots}
\pgfplotsset{compat=newest}
\usepackage{amsthm}
\usepgfplotslibrary{patchplots}

\crefname{figure}{Fig.}{Fig.}

\newtheorem{thm}{Theorem}
\crefname{thm}{Theorem}{Theorems}

\newtheorem{prop}{Proposition}
\crefname{prop}{Proposition}{Propositions}
\newtheorem{lem}{Lemma}
\crefname{lem}{Lemma}{Lemmas}
\newtheorem{cor}{Corollary}
\crefname{cor}{Corollary}{Corollaries}
\theoremstyle{remark}
\newtheorem{rem}{Remark}
\crefname{rem}{Remark}{Remarks}

\theoremstyle{definition}

\crefname{example}{Example}{Examples}

\crefname{ass}{Assumption}{Assumption}
\usepackage{dsfont}
\let\mathbb=\mathds
\usepackage[numbers,sort&compress]{natbib}

\crefname{conj}{Conjecture}{Conjectures}

\theoremstyle{definition}
\newtheorem{defn}{Definition}
\crefname{defn}{Definition}{Definitions}

\crefname{prob}{Problem}{Problems}
\crefname{algorithm}{Algorithm}{Algorithms}

\newcommand{\Rmn}{\mathbb{R}^{n \times m}}

\newcommand{\Rnn}{\mathbb{R}^{n \times n}}

\newcommand{\rk}{\textnormal{rank}}

\newcommand{\diag}{\textnormal{diag}}
\newcommand{\sign}{\textnormal{sign}}

\newcommand{\opts}{\star}

\newcommand{\transp}{\mathsf{T}}
\newcommand{\vari}[1]{\text{S}(#1)}

\newcommand{\tp}[1]{{#1}\text{-positive}}
\newcommand{\vardim}[1]{{#1}\text{-variation diminishing}}

\newcommand{\ovd}[1]{$\text{OVD}_{#1}$}

\newcommand{\Con}[1]{{\mathcal{C}^{#1}}}
\newcommand{\Obs}[1]{{\mathcal{O}^{#1}}}

\newcommand{\Toep}[1]{\mathcal{T}_{#1}}
\newcommand{\Hank}[1]{\mathcal{H}_{#1}}

\newcommand{\linf}{\ell_\infty}

\newcommand{\compound}[2]{#1_{[#2]}}

\colorlet{FigColor1}{blue}
\colorlet{FigColor2}{red}
\colorlet{FigColor3}{ao(english)}
\colorlet{FigColor4}{orange}
\pgfplotsset{every axis plot/.append style={line width=1.5pt}}

\crefformat{equation}{\textup{#2(#1)#3}}
\crefrangeformat{equation}{\textup{#3(#1)#4--#5(#2)#6}}
\crefmultiformat{equation}{\textup{#2(#1)#3}}{ and \textup{#2(#1)#3}}
{, \textup{#2(#1)#3}}{, and \textup{#2(#1)#3}}
\crefrangemultiformat{equation}{\textup{#3(#1)#4--#5(#2)#6}}%
{ and \textup{#3(#1)#4--#5(#2)#6}}{, \textup{#3(#1)#4--#5(#2)#6}}{, and \textup{#3(#1)#4--#5(#2)#6}}

\Crefformat{equation}{#2Equation~\textup{(#1)}#3}
\Crefrangeformat{equation}{Equations~\textup{#3(#1)#4--#5(#2)#6}}
\Crefmultiformat{equation}{Equations~\textup{#2(#1)#3}}{ and \textup{#2(#1)#3}}
{, \textup{#2(#1)#3}}{, and \textup{#2(#1)#3}}
\Crefrangemultiformat{equation}{Equations~\textup{#3(#1)#4--#5(#2)#6}}%
{ and \textup{#3(#1)#4--#5(#2)#6}}{, \textup{#3(#1)#4--#5(#2)#6}}{, and \textup{#3(#1)#4--#5(#2)#6}}

\crefdefaultlabelformat{#2\textup{#1}#3}

\begin{document}

\title{Variation diminishing linear time-invariant systems}
\author{Christian Grussler \thanks{The author is with the University of California Berkeley, Berkeley, CA 94720, USA, 	{\tt\small christian.grussler@berkeley.edu}.} and Rodolphe Sepulchre%
	\thanks{The author is with the Control Group at the Department of Engineering, University of Cambridge, Cambridge CB2 1PZ, United Kingdom,
		{\tt\small r.sepulchre @eng.cam.ac.uk}.}}%

\maketitle

	\begin{abstract}	
This paper studies the variation diminishing property of $k$-positive linear time-invariant (LTI) systems, which map inputs with $k-1$ sign changes to outputs with at most the same variation. We characterize this property for the Toeplitz and Hankel operators of finite-dimensional systems. Our main result is that these operators have a dominant approximation in the form of series or parallel interconnections of $k$ first order positive systems. This is shown by expressing the $k$-positivity of a LTI system as the external positivity (that is, $1$-positivity) of $k$ {\it compound} LTI systems. Our characterization generalizes well known properties of externally positive systems ($k=1$) and totally positive systems ($k=\infty$; also known as relaxation systems). 
\end{abstract}

\section{Introduction}
Positive systems, that is, models that map positive inputs to positive outputs, have gained considerable interest in the recent years, e.g., \cite{farina2011positive,rantzer2015scalable,son1996robust,tanaka2011bounded}. They appear frequently in networks, economics, biology, transport, etc., and are attractive for their favourable analytical properties. Positivity provides computational scalability in many standard control problems such as  Lyapunov analysis \cite{rantzer2015scalable}, optimal control design \cite{tanaka2011bounded}, or system gain computation \cite{farina2011positive}. Specific types of positive systems such as the parallel interconnection of first order lags have already been studied in the context of relaxation systems and passivity \cite{willems1976realization,pates2019optimal}.

Our goal in the present paper is to connect positive systems to the classical theory of total positivity,  where positivity has been studied as a {\it variation diminishing} property. This approach has a long history in statistics and mathematics \cite{karlin1968total,fekete1912uber,Schoenberg1930vari,Schoenberg1951polya}. Variation diminishing properties are expected from any reasonable smoothing operation: a smoothing filter should not output a signal with more irregularities than the input signal. In the theory of total positivity, $k$-positivity refers to the variation diminishing property of inputs with at most $k-1$ variations. External (input-output) positivity ($k=1$) and total positivity ($k=\infty$) are the two extremes of a hierarchical structure. While characterizations for the extreme cases $k=1$ and $k=\infty$ have been well studied, analogues for the cases $1 < k < \infty$, even in the simplest situation of finite-dimensional linear time-invariant (LTI) systems, are missing. The objective in this paper is to study the $k$-positivity of the Hankel and Toeplitz operators of LTI systems. We build upon the classical theory of Karlin \cite{karlin1968total} to provide the following results: first, $k$-positivity of a positive LTI system is equivalent to external positivity (that is, $1$-positivity) of $k$ compound LTI systems. This is important as it provides us with computational and theoretical tractability; second, the positivity of the compound systems implies that the dominant dynamics of a $k$-positive system has a simple decomposition in terms of $k$ first-order positive systems: parallel/serial interconnection of first order lags for the Hankel/Toeplitz operator. In other words, the $k$-th order dominant approximation of a $k$-positive system is totally positive with respect to its operator structure.

The interesting special case of systems with $2$-positive Toeplitz operator has been studied in preliminary work \cite{grussler2018strongly}. Such systems map unimodal (single-peaked) inputs to unimodal outputs, which explains their particular importance in statistics, as kernels that preserve the unimodality of many important distributions \cite{karlin1968total,Schoenberg1951polya,ibragimov1956composition} (see~Figure~\ref{fig:UNIMOD}).

\tikzstyle{int}=[draw,minimum width=1cm, minimum height=1cm, very thick, align = center]
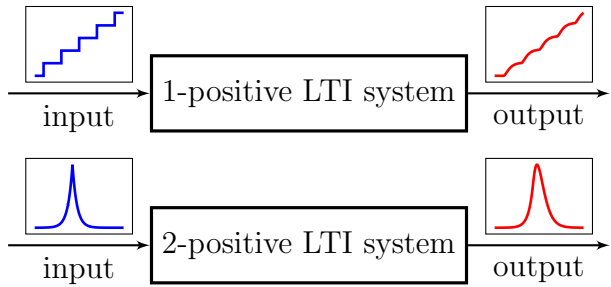
\begin{figure}
	\hspace*{0.02 cm}
	\begin{center}
		\begin{tikzpicture}[>=latex']
		\node (a) [int] {$1$-positive LTI system}; 
		\node (b) [left of=a,node distance=4cm, coordinate]{};
		\node [coordinate] (end) [right of=a, node distance=4cm]{};
		
		\path[->,thick] (b) edge node[below]{input} node(u_inc)[above]{} (a);
		\begin{axis}[ticks = none,width = 3 cm,at=(u_inc), anchor={south}]
		\addplot[line width = 1 pt, color = FigColor1] file{u_inc1.txt};
		\end{axis}
		
		\path[->,thick] (a) edge node[below]{output} node(y_inc)[above,midway]{} (end);
		\begin{axis}[ticks = none,width = 3 cm,at=(y_inc), anchor={south}]
		\addplot[line width = 1 pt, color = FigColor2] file{y_inc1.txt};
		\end{axis}
		
		\node (a1) [int, below of = a, anchor = north, yshift = - .5 cm] {$2$-positive LTI system}; 
		\node (b1) [left of=a1,node distance=4cm, coordinate]{};
		\node [coordinate] (end1) [right of=a1, node distance=4cm]{};
		
		\path[->,thick] (b1) edge node[below]{input} node(u_inc)[above]{} (a1);
		\begin{axis}[ticks = none,width = 3 cm,at=(u_inc), anchor={south}]
		\addplot[line width = 1 pt, color = FigColor1] file{u_uni1.txt};
		\end{axis}
		
		\path[->,thick] (a1) edge node[below]{output} node(y_inc)[above,midway]{} (end1);
		\begin{axis}[ticks = none,width = 3 cm,at=(y_inc), anchor={south}]
		\addplot[line width = 1 pt, color = FigColor2] file{y_uni1.txt};
		\end{axis}
		
		\end{tikzpicture}
	\end{center}
	\caption{A 1-positive (externally positive) LTI system maps monotone inputs to monotone outputs. A $2$-positive LTI system maps unimodal inputs to unimodal outputs.}
	\label{fig:UNIMOD}
\end{figure}

The proposal in that paper is that the variation diminishing property is an important system property that is amenable to a tractable analysis and open new analysis avenues for nonlinear systems. While the paper primarily focuses on LTI systems, we suggest in the example section that the analysis extends to cascades of LTI systems with static non-linearities. Such structures are frequently used in machine learning and in biology to model non-linear filters and fading memory operators. We envision that the use of positivity to characterize the variation diminishing, that is, smoothing, property of such non-linear filters could find promising applications in the analysis of large-scale interconnections of such basic elements.

The remainder of the paper is organized as follows. After some preliminaries in Section~\ref{sec:prelim}, the theory of variation diminishment and total positivity is reviewed in Section~\ref{sec:vardim}. In Section~\ref{sec:hankel}, we present our main results for the Hankel operator case, where also the notion of compound systems is introduced and discussed. Then, in Section~\ref{sec:toep}, we briefly discuss analogous for the Toeplitz operator. The paper ends with a presentation of examples in Section~\ref{sec:app} and concluding remarks in Section~\ref{sec:conc}. Proofs are given in the appendix. Our exposition is limited to single-input single-output discrete-time systems, but a similar treatment for continuous-time  systems can be applied.

\section{Preliminaries}\label{sec:prelim}
	\subsection{Notations}
	\subsubsection{Sets}
 For $\mathcal{S} \in \mathds{R}$, we use $\mathcal{S}_{\geq k} := \mathcal{S} \cap [k, \infty)$, e.g., $\mathds{R}_{\geq 0} = [0,\infty)$ and $\mathds{Z}_{\geq 0} = \mathds{N}_{0}$. Further, for $k, l \in \mathds{Z}$, we use $(k:l) := \{k,k+1,\dots,l\}$ if $k \leq l$ and $(k:l) := \{k,k-1,\dots,l\}$ if $l \leq k$. We define the $i-th$ elements of the $r$-tuples in $\mathcal{I}_{n,r} := \{ v = \{v_1,\dots,v_r\} \subset \mathds{N}: 1\leq v_1 <  \dots < v_r \leq n \}$
	 by \emph{lexicographic ordering}, i.e., for $\bar{v}, \tilde{v} \in \mathcal{I}_{n,r}$ it holds that $\bar{v} < \tilde{v}$ if and only if $\bar{v}_{i^\ast} < \tilde{v}_{i^\ast}$, $i^\ast := \min \{i: \bar{v}_i \neq \tilde{v}_i \}$. For example, $\mathcal{I}_{4,3}$ is sorted in the order $(1,2,3)$, $(1,2,4)$, $(2,3,4)$.
	\subsubsection{Matrices}
	For real valued matrices $X = (x_{ij}) \in \Rmn$, including vectors $x = (x_i) \in \mathbb{R}^n$, we say that $X$ is \emph{nonnegative}, $X \geq 0$ or $X \in \Rmn_{\geq 0}$, if all elements $x_{ij} \in \mathbb{R}_{\geq 0}$. Further, if $X \in \Rnn$, then $\sigma(X) = \{\lambda_1(X),\dots,\lambda_n(X)\}$ denotes its \emph{spectrum}, where the eigenvalues are ordered by descending absolute value, i.e., $\lambda_1(X)$ is the eigenvalue with the largest magnitude, counting multiplicity. In case that the magnitude of two eigenvalues coincides, we subsort them by decreasing real part. 
	The \emph{identity matrix} in $\Rnn$ is denoted by $I_n$. A matrix $X \in Rnn$ is said to be \emph{positive semidefinite}, $X \succeq 0$, if $X = X^\transp$ and $\sigma(X) \subset \mathbb{R}_{\geq 0}$. By Sylvester's criterion \cite[Theorem~7.2.5]{horn2012matrix} for positive definiteness
	\begin{equation}
	X\succ 0 \; \Leftrightarrow \; \det(X_{(1:j),(1:j)}) > 0, \; 1\leq j\leq n, \label{eq:psd}
	\end{equation}
	where \emph{submatrices} are written as $X_{I,J} := (x_{ij})_{i \in I, j \in J}$ with $I \subset (1:n)$ and $J \subset (1:m)$. If $I$ and $J$ have both cardinality $r$, then $\det(X_{I,J})$ is also referred to as an \emph{$r$-minor}. A minor is called \emph{consecutive}, if $I$ and $J$ are intervals. The so-called \emph{Desnanot-Jacobi identity}
	\begin{equation}
		\det(X) \det(X_{(2:n-1),(2:n-1)})  = \det\begin{pmatrix}
	\det(X_{(1:n-1),(1:n-1)}) & \det(X_{(1:n-1),(2:n)})\\
	\det(X_{(2:n),(1:n-1)}) & \det(X_{(2:n),(2:n)})
	\end{pmatrix} \label{eq:Sylvester_id}
	\end{equation}
	shows how the determinant can be computed from consecutive minors. The \emph{$r$-th multiplicative compound matrix} $\compound{X}{r}$ of $X \in \Rmn$ is defined as the matrix that is made of all $r$-minors and whose $(i,j)$-th entry is defined by $\det(X_{(I,J)})$, where $I$ is the $i$-th and $J$ is the $j$-th element in $\mathcal{I}_{n,r}$ and $\mathcal{I}_{m,r}$, respectively. For example, if $X \in \mathbb{R}^{3 \times 3}$, then $\compound{X}{r}$ reads
\begin{align*}
\begin{pmatrix}
\det(X_{\{1,2 \},\{1,2 \}}) & \det(X_{\{1,2 \},\{1,3\}}) & \det(X_{\{1,2 \},\{2,3\}})\\
\det(X_{\{1,3 \},\{1,2 \}}) & \det(X_{\{1,3 \},\{1,3\}}) & \det(X_{\{1,3 \},\{2,3\}})\\
\det(X_{\{2,3 \},\{1,2 \}}) & \det(X_{\{2,3 \},\{1,3\}}) & \det(X_{\{2,3 \},\{2,3\}})\\
\end{pmatrix}.
\end{align*}
By the \emph{Cauchy-Binet formula} \cite{karlin1968total} with $X \in \mathbb{R}^{n \times p}$ and $Y \in \mathbb{R}^{p \times m}$
\begin{equation}
\det((XY)_{I,J}) = \sum_{K \in \mathcal{I}_{p,r}} \det(X_{I,K}) \det(Y_{K,J}), \label{eq:CB}
\end{equation}
where $I \in \mathcal{I}_{n,r}$, $J \in \mathcal{J}_{m,r}$ and $r \leq p$, it is easy to verify the following properties \cite[Chapter~6]{fiedler2008special}.
\begin{lem}\label{lem:compound_mat}
	Let $X \in \mathbb{R}^{n \times p}$, $Y \in \mathbb{R}^{p \times m}$ and $r \in \mathds{Z}_{\geq 1}$.
	\begin{enumerate}[i)]
		\item $C_{[r]}(XY) = \compound{X}{r}\compound{Y}{r}$.
		\item $\sigma(\compound{X}{r}) = \{\prod_{i \in I} \lambda_i(X): I \in \mathcal{I}_{n,r} \}$.
		\item If $X \succeq 0$, then $\compound{X}{r} \succeq 0$.
	\end{enumerate} 
\end{lem}
Finally, $X \in \mathbb{R}^{n \times n}$ is referred to as \emph{Toeplitz (Hankel, respectively)} if $X$ is constant along its (anti-, respectively) diagonals.

	\subsubsection{Functions}
	In the following, we introduce several notations for functions $g: \mathds{Z} \to \mathbb{R} \cup \{\pm \infty\}$. %
	We write $g \geq 0$ for a \emph{nonnegative function} $g: \mathds{Z} \to \mathbb{R}_{\geq0}$ and $g(i:j) := \begin{pmatrix}
g(i) & \dots & g(j)
\end{pmatrix}^\transp$ for snapshots. Further, for \emph{compositions} with $\sigma: \mathds{R} \to \mathds{R}$, define $\sigma(g)(t) := \sigma(g(t))$. 
By denoting the \emph{(1-0) indicator function} of $\mathcal{S} \subset \mathds{Z}$ as
	\begin{align*}
	\mathds{1}_{\mathcal{S}}(t) := \begin{cases}
	1 & t \in \mathcal{S}\\
	0 & t \notin \mathcal{S}
	\end{cases},
	\end{align*}
	we express the \emph{Heaviside function} as $s(t) := \mathds{1}_{\mathbb{R}_{\geq 0}}(t)$ and the \emph{unit pulse function} as $\delta(t) := \mathds{1}_{\lbrace 0 \rbrace}(t)$. 
		The $k$-th \emph{forward difference} of $g: \mathds{Z} \to \mathbb{R}$ is abbreviated by $\Delta^{(k)} g(t) := \Delta^{(k-1)} g(t+1) - \Delta^{(k-1)}g(t)$ with $\Delta^{(1)} g(t) := g(t+1) - g(t)$. %
		
		We call $g: \mathds{Z} \to \mathbb{R} \cup \{\infty\}$ \emph{convex} if for all $t_1,t_2 \in \mathds{Z}$ and $0 \leq \lambda \leq 1$ it holds that
		\begin{align}
		\label{eq:def_cvx}
		\lambda g(t_1) + (1-\lambda)g(t_2) &\geq \min_{u \in \mathcal{B}(z) } g(u) %
		\end{align}
where $z := \lambda t_1 + (1-\lambda)t_2$ and $\mathcal{B}(z) := \{t \in \mathds{Z}: |z-t| < 1 \}$. $g$ is convex if and only if $\Delta^{(2)} g(t) \geq 0$ for all $t$ \cite{yuccer2002discrete}. %
If $-g$ is convex, then $g$ is called \emph{concave}. Further, for an interval $\mathcal{S} \subset \mathds{Z}$ we call $g: \mathcal{S} \to \mathds{R}_{\geq 0} \cup \{\infty\}$ \emph{logarithmically (log-)convex} on $\mathcal{S}$ if $\log(g)$ is convex or equivalently $g(t)g(t+2) - g(t+1)^2 \geq 0$ for all $t \in \mathcal{S}$. Analogously, $g$ is \emph{logarithmically (log-)concave} if $\frac{1}{g}$ is log-convex or  $g(t+1)^2-g(t)g(t+2)\geq 0$ for all $t \in \mathcal{S}$. The sets of all bounded functions is denoted by $\linf$.

\subsection{Linear discrete time systems}
We consider finite dimensional {causal linear discrete time-invariant (LTI) systems} with (scalar) input $u$ and (scalar) output $y$. The \emph{impulse response} is the output corresponding to the input $u(t) = \delta(t)$. The \emph{transfer function} of the system is
\begin{equation}
G(z) = \sum_{t=0}^\infty g(t)z^{-t} = r \frac{ \prod_{i=1}^{m} z-z_i}{\prod_{j=1}^{n}z-p_i},
\end{equation}
where $m \leq n$, $r \in \mathbb{R}$, $p_i$ and $z_i$ are referred to as \emph{poles} and \emph{zeros}, which are both sorted in same way as the eigenvalues of a matrix. For national convenience, we assume that all systems are non-zero and \emph{strictly proper}, i.e., $m < n$. The tuple $(A,b,c)$ is then referred to as a \emph{state-space realization} if
\begin{equation}\label{eq:SISO_d}
\begin{aligned}
x(t+1) = Ax(t) + bu(t), \quad y(t) = cx(t), 
\end{aligned} 
\end{equation}
with $A \in \mathbb{R}^{n\times n}$, $b, c^\transp \in \mathbb{R}^n$. It holds then that \begin{equation}
g(t) = cA^{t-1}bs(t-1)
\end{equation}
 In this work, we assume that $g,u,y \in \linf$, that the poles and zeros of a transfer function are disjoint and that $G(z) \not \equiv 0$. If the same applies to $c(zI_n-A)^{-1}b$, then $(A,b,c)$ is called minimal and $G(z) = c(zI_n-A)^{-1}b$. 
Further, if the system has simple poles then by the partial fraction decomposition
$G(z) = \sum_{i=1}^{n} \frac{r_i}{z-p_i}$. A system $G(z)$ that has all poles at zero is called a \emph{finite impulse response (FIR)} system.

The \emph{Hankel} operator 
	\begin{align*}
	(\Hank{g} u)(t) &:= \sum_{\tau=-\infty}^{-1} g(t-\tau)u(\tau) = \sum_{\tau=1}^{\infty} g(t+\tau)u(-\tau),\ t \geq 0 %
	\end{align*}

describes the output $y(t)$ corresponding to the input $u(t) = u(t)(1-s(t))$. It is a mapping from past inputs $(t \in (-\infty, 0)$ to future outputs ($t \geq 0$) under the convolution with the impulse response. The \emph{Toeplitz operator}
	\begin{align*}
	(\Toep{g} u)(t) &:= \sum_{\tau=0}^{t} g(t-\tau)u(\tau), \ t \geq 0 %
	\end{align*}  
maps a future input $u(t)$, $t \ge 0$, to the corresponding output $y(t)$, $t \ge 0$. In a state-space model, the Hankel operator models the free response of the system for an initial condition $x(0)$ (that parametrizes the past input), while the Toeplitz operator models the input-output behavior under the assumption that $x(0) = 0$. By defining for $j \geq 1$\begin{subequations}
\begin{equation*}%
H_g(t,j) := \begin{pmatrix}
g(t) & g(t+1) & \dots & g(t+j-1)\\
g(t+1) & g(t+2) & \dots & g(t+j)\\
\vdots & \vdots & \ddots   & \vdots \\
g(t+j-1)   & g(t+j) & \dots & g(t-2(j-1))\\
\end{pmatrix}
\end{equation*}
with $t\geq 1$ and
\begin{equation*}%
T_g(t,k) := \begin{pmatrix}
g(t) & g(t-1) & \dots & g(t-k+1)\\
g(t+1) & g(t) & \dots & g(t-j)\\
\vdots & \vdots & \ddots   & \vdots \\
g(t+k-1)   & g(t+k-2) & \dots & g(t)\\
\end{pmatrix}
\end{equation*}
with $t \geq 0$,
\end{subequations}
these operators can be represented by the finitely truncated matrix representations $\Hank{g}^ju := H_g(1,j) u(-1:-j)$ and $\Toep{g}^ju := T_g(0,j) u(0:j),$ as
$\Hank{g}u = \lim_{j \to \infty} \Hank{g}^ju$ and $\Toep{g}u = \lim_{j \to \infty} \Toep{g}^ju$. In particular, for a state-space realization $(A,b,c)$
\begin{equation}
H_g(t,j) = \Obs{j}(A,c) A^{t-1}\Con{j}(A,b) \label{eq:Hankel_mat_real}
\end{equation}
with the extended controllability and observability matrices
\begin{subequations}
	\begin{align}
	\Con{j}(A,b) &:= \begin{pmatrix}
	b & Ab & \dots & A^{j-1} b
	\end{pmatrix},\\
	\Obs{j}(A,c) &:= \begin{pmatrix}
	c^\transp & A^\transp c^\transp & \dots & {A^\transp}^{j-1} c^\transp
	\end{pmatrix}^\transp.
	\end{align}
\end{subequations}
If $(A,b,c)$ is a minimal realization, it holds then that $\rk(H_g(1,j)) = j$ for $j \leq n$.

\section{The Variation Diminishing Property}
\label{sec:vardim}
The variation of a sequence or vector $u$ is defined as the number of sign changes in $u$, i.e.,
$$\vari{u} := \sum_{i \geq 1} \mathbb{1}_{\mathbb{R}_{< 0}}(\tilde{u}_i \tilde{u}_{i+1}), \; S(0) := 0$$
where  $\tilde{u}$ is obtained from deleting all zero entries in $u$.
A linear map $X$ is said to be \emph{order preserving $k$-variation diminishing (\ovd{k})}, $k \in \mathds{Z}_{\geq 0}$, if for all $u$ with $\vari{u} \leq k$ it holds that
\begin{enumerate}[i.]
    \item $\vari{Xu} \leq \min\{\rk(X)-1,\vari{u}\}$.
    \item The sign of the first non-zero elements in $u$ and $Xu$ coincide whenever $\vari{u} = \vari{Xu}$.
\end{enumerate}
If the second item is dropped, $X$ is simply called \emph{\vardim{k}}. Investigations of such linear such mappings have a long history. They include the determination of real poles \cite{fekete1912uber}, applications in interpolation theory \cite{Schoenberg1951polya,karlin1968total}, vibrational systems \cite{gantmacher1950oszillationsmatrizen}, computer vision \cite{lindeberg1990scale} as well as bounding the number of sign-changes in an impulse response \cite{elkhoury1993discrete,elkhoury1993influence}. The theory applies to finite dimensional matrices as in \cite{Schoenberg1930vari,karlin1968total} as well as to convolution operators \cite{karlin1968total,Schoenberg1951polya,Samworth2017,ibragimov1956composition}. 

The goal of this paper is to provide a characterization of \ovd{k} for finite-dimensional LTI systems. We study separately the variation diminishing property for the Hankel operator and for the Toeplitz operator. We start our investigations by reviewing the most well-known cases of $k=0$, $k=1$ and $k = \infty$. The case $k=0$ will be instrumental for our new developments from which the other cases can be recovered as part of our general analysis. We conclude this section by a short review of \emph{total positivity theory}.

\subsection{\ovd{0}-systems}
A Toeplitz \ovd{0}-system $G(z)$ maps zero variation inputs $u \geq 0$ to zero variation outputs $y \geq 0$, which coincides with the definition of so-called \emph{externally positive systems} \cite{Farina2000,ohta1984reachability}. Since $u(t) = \delta(t) \geq 0$, it equivalently holds that $g(t)\geq 0$.  An important frequency domain property of an externally positive systems is their dominant approximation in form of a first order lag. 
\begin{prop}
	\label{prop:dominantpole}
	Let $G(z)$ be an externally positive system with dominant pole $p_1$. Then, $p_1 \geq 0$ and all real zeros of the system are smaller than $p_1$. Further, if $G(z) = \sum_{i\geq1 }\frac{r_{i}}{z-p_i}$, then $r_{1} > 0$.
\end{prop} 
In particular, every first order lag $G(z) = \frac{r_1}{z-p_1}$ with $r_1,p_1 \geq 0$ is externally positive. Note that the verification of external positivity is in general NP-hard \cite{blondel2002presence}, but under mild conditions good sufficient certificates \cite{Farina2000,grussler2019tractable,drummond2019external} are available. 

\subsection{\ovd{1}-systems} \label{subsec:2_dim}
By definition, \ovd{1} systems are \ovd{0} systems that map inputs $u$ with one sign change to outputs $y$ with at most one sign change, where the order of the sign changes coincide if $\vari{u} = \vari{y} = 1$. In the statistics literature, \ovd{1} integral kernels are called \emph{strongly unimodal} \cite{karlin1968total,Schoenberg1951polya,Samworth2017,ibragimov1956composition}. This terminology reflects the fact that due to linearity, one can equivalently consider the variation diminishment of the forward differences, meaning that unimodal ($\vari{\Delta^{(1)}u} \leq 1$) inputs are mapped to unimodal outputs. As the inputs $\delta(t)$ and $\delta(t+1)$ are unimodal, it follows that $g$ needs to be unimodal for $G(z)$ to be Toeplitz and Hankel \ovd{1}, respectively. 

Unfortunately, this is not a sufficient characterization. On the one hand, the Hankel case requires that $g$ is log-convex on $\mathds{Z}_{> 0}$. In particular, this means that $g$ is convex unimodal and therefore monotonically decreasing \cite{boyd2004convex}. On the other hand, the Toeplitz case needs $g$ to be log-concave on $\mathds{Z}_{\geq 0}$ \cite{karlin1968total}. The only systems that are both log-convex and -concave are first order ones, where the inequalities hold with equality. To illustrate these differences, consider 
\begin{align*}
G_1(z) = \frac{r_1}{z-p_1} \frac{r_2}{z-p_2},  \quad G_2(z) = \frac{r_1}{z-p_1} + \frac{r_2}{z-p_2}
\end{align*}
with $r_1,r_2, p_1,p_2 > 0$. $G_1(z)$ is Toeplitz \ovd{1} as the parallel interconnection of Toeplitz \ovd{1} systems. However, as a second order system, $g_1$ cannot be log-convex. The opposite holds for $G_2(z)$ as log-convexity is preserved under summation (see~Lemma~\ref{lem:intercon}).

Our results will show that those basic transfer functions are somewhat fundamental. In particular, Toeplitz \ovd{1} requires dominant dynamics in the form of $G_1(z)$  while Hankel \ovd{1} requires dominant dynamics in the form of $G_2(z)$.
\subsection{Totally \ovd{}-systems}
Totally \ovd{}-systems diminish the variation from all inputs $u$ to outputs $y$ and preserve the sign change order whenever $\vari{u} = \vari{y}$. The statistics literature \cite{karlin1968total,Aissen1952generating} offers the following frequency domain characterization.
\begin{prop}\label{prop:toep_inf_pos}
For $G(z)$ it holds that
\begin{enumerate}[i)]
	\item $G(z)$ is Toeplitz totally \ovd{} if and only if
	$G(z) =  \prod_{i=1}^n \frac{r_i z+\alpha_i}{z-p_i}$
	$\alpha_i, r_i, p_i \geq 0$, i.e., it is the serial interconnection of first order lags with negative zeros.
	\item $G(z)$ is totally Hankel totally \ovd{} if and only if $G(z) = \sum_{i=1}^{n} \frac{r_i}{z-p_i}$,
	where $r_i , p_i \geq 0$, i.e., it is the parallel interconnection of first order lags.
\end{enumerate}
\end{prop}
Hankel totally positive systems are also considered  in \cite{willems1976realization} under the name of \emph{relaxation systems}.
\begin{prop}\label{prop:relax}
   The following are equivalent:
    \begin{enumerate}[i.]
        \item $G(z) = \sum_{i=1}^{n} \frac{r_i}{z-p_i}$ with $r_i , p_i \geq 0.$
        \item $H_g(1,n) \succ 0$ and $H_g(2,n) \succeq 0$.
        \item $(-1)^j \Delta^{(j)}g \geq 0$ for all $j \in \mathds{Z}_{\geq 0}$
    \end{enumerate}
\end{prop}
Note that the Hankel operator results have been stated in continuous-time in the literature. Proofs for the discrete time case are provided in this paper. 

\subsection{$k$-positive matrices}
Finite-dimensional \ovd{k-1} linear operators (matrices) are a classical subject of matrix theory, e.g., extensively studied in \cite{karlin1968total}.  
\begin{defn}[$k$-positivity]
\label{def:sc}
	A matrix $X \in \Rmn$ is called \emph{(strictly) \tp{k}} with $k \in \mathds{Z}_{>0}$ if all $j$-minors are (positive) nonnegative for $1\leq j \leq k$. If $k = \min\{m,n\}$, we simply say \emph{(strictly) totally positive}. 
\end{defn}
The relationship between \ovd{k-1} and the above definitions is provided by the following result, proven in  Appendix~\ref{proof:prop:SR}.
\begin{prop} \label{prop:SR}
	Let $X \in \Rmn$ with $n \geq m$. Then $X$ is $\tp{k}$ with $k \leq m$ if and only if $X$ is \ovd{k-1}. %
\end{prop}
The reader will notice that the number of $k$ minors of $X$ is combinatorial. Fortunately, it often suffices to only consider consecutive minors \cite{karlin1968total,fallat2017total}.
\begin{prop}
    \label{prop:consecutive}
	If $X \in \Rmn$, $k \leq \min \{n,m\}$ is such that all consecutive $j$-minors of $X$ are positive, $1 \leq j \leq k-1$ and all consecutive $k$-minors of $X$ are nonnegative (positive). Then, $X$ is (strictly) \tp{k}.
\end{prop}

\section{Hankel $k$-positivity}
\label{sec:hankel}
In this section, we provide a characterization of Hankel \ovd{k-1} systems for arbitrary $k$. We begin by stating our main results based on the following definition.
\begin{defn}[Hankel $k$-positivity] \label{def:hankel_kpos}
$G(z)$ is called \emph{Hankel (strictly) $\tp{k}$} if the Hankel matrix $H_g(1,N)$ is (strictly) $\tp{N}$ for all $N \geq k$. We say that $G(z)$ is \emph{Hankel (strictly) totally positive} if $k = \infty$.   
\end{defn}
 Our main tool for the analysis is the concept of \emph{$j$-th compound system}.  
\begin{defn}[Compound Systems]
	For $G(z)$ and $j \in \mathds{Z}_{\geq 1}$, we define the $j$-th compound system $G_{[j]}(z)$ by the impulse response $g_{[j]}(t) := \det(H_g(t,j))$, $t \geq 1$.
\end{defn}
Our first  main result is a characterization of Hankel \ovd{k} in terms of Hankel $k$-positivity through the compound systems.
\begin{thm} \label{prop:con_minor}
	If $G(z)$ is an $n$-th order system with $k \leq n$, then the following are equivalent:
		\begin{enumerate}[i.]
		\item $G(z)$ is Hankel \ovd{k-1}. \label{item:Hankel_vardim}  
		\item $G(z)$ is Hankel $\tp{k}$. 
		\item The compound systems $G_{[j]}(z)$ are externally positive for $1 \leq j \leq k$. \label{item:Hankel_comp}
		\item $H_g(1,k-1) \succ 0$, $H_g(2,k-1) \succeq 0$ and $G_{[k]}(z)$ is externally positive. 
		\item $G_{[j]}$ is Hankel $\tp{k-j+1}$ for $1\leq j \leq k$.
	\end{enumerate} 
\end{thm}
The distinction between \ovd{k-1} and $\tp{k}$ emphasizes that the former is an input-output property, while the latter is a matrix property. Our analysis of the compound system properties yields the second main result: a necessary frequency domain characterization in terms of dominant approximations. 
\begin{thm}\label{thm:decomp_tpk}
	Let $G(z) =  \sum_{i=1}^{n} \frac{r_i}{z-p_i}$ have distinct poles and $2 \leq k \leq n$. If $G(z)$ is Hankel $\tp{k}$, then $G(z) = \frac{r_{1}}{z-p_1} + G_r(z)$, where $G_r(z)$ is Hankel $\tp{k-1}$ with $r_{1} > 0$.
\end{thm}
By applying Theorem~\ref{thm:decomp_tpk} recursively to $G_r(z)$, it follows that a Hankel $\tp{k}$ system has the decomposition 
\begin{equation*}
    G(z) = \sum_{i=1}^{k-1} \frac{r_i}{z-p_i} + G_r(z)
\end{equation*}
with externally positive $G_r(z)$. By Propositions~\ref{prop:dominantpole} and \ref{prop:toep_inf_pos},  the dominant approximation $\sum_{i=1}^{k} \frac{r_i}{z-p_i}$ is Hankel totally positive, which demonstrates the bridge between the extreme cases $k=1$ and $k=\infty$: Hankel $\tp{k}$ systems have a $k$-th order dominant approximation that is totally positive.

\subsection{Hankel total positivity theory}
This subsection is devoted to the proof of Theorem~\ref{prop:con_minor}.
\begin{lem}\label{prop:op_seq}
	For $G(z)$ and $k \in \mathbb{Z}_{\geq 1}$, the following are equivalent:
	\begin{enumerate}[i)]
		\item $G(z)$ is Hankel \ovd{k-1}.  \label{item:op_var_dim}
		\item $\forall j \in \mathds{Z}_{\geq 0}: \; \Hank{g}^ju$ is \ovd{k-1} for  all $u$ such that $\Hank{g}^ju \in \linf$ \label{item:sequence_var_dim}
	\end{enumerate}
\end{lem}
Note that $\Hank{g}^ju$ can be finitely represented by $H_g(1,j)$, which by its repetitive structure is $\tp{k}$ if and only if $H_g(1,j)_{(1:j,1:\min\{j,k\})}$ is $\tp{k}$. Proposition~\ref{prop:SR} yields then the equivalence between the first and second item in Theorem~\ref{prop:con_minor}. Further, since all $i$-minors of $H_g(1,j)$ are symmetric, it follows from (\ref{eq:psd}) that $G(z)$ is strictly Hankel $\tp{k}$ if and only if all $i$-minors of $H_g(1,j)$, $1\leq i \leq k$, are positive definite for all $j \geq k$. As shown in \cite{fallat2017total}, the set of strictly Hankel $\tp{k}$ systems then inherits the properties of positive definite matrices. 
\begin{lem}\label{lem:intercon}
For $k \in \mathds{Z}_{>0}$, let $\mathcal{(S)HP}_k$ denote the set of all (strictly) Hankel $\tp{k}$ systems. Then, the following hold:
\begin{enumerate}[i.]
    \item $\mathcal{HP}_k$ is a proper convex cone. 
    \item If $G_1(z) \in \mathcal{SHP}_{k_1}$ and $G_2(z) \in \mathcal{(S)HP}_{k_2}$, then \begin{enumerate}
        \item $G_1(z)+G_2(z) \in \mathcal{SHP}_{\min \{ k_1,k_2 \}}$.
        \item $G_P(z) \in \mathcal{(S)HP}_{\min \{ k_1,k_2 \}}$ with $g_p = g_1g_2$.
    \end{enumerate} 
    \item $\mathcal{SHP}_{k}$ is dense in $\mathcal{HP}_{k}$. 
    \item $G_k(z) = \sum^{k}_{i =1}\frac{k_i}{z-p_i} \in \mathcal{SHP}_{k} \cap \mathcal{HP}_{\infty}$ if all $k_i,p_i > 0$ and $p_i$ are distinct. \end{enumerate}
\end{lem}
It follows from Lemma~\ref{lem:intercon} that $G(z)$ is Hankel $\tp{k}$ if and only if $G(z)+\varepsilon G_k(z)$ is strictly Hankel $\tp{k}$ for all $\varepsilon >0$. By continuity, it thus suffices to prove our results under strict Hankel $k$-positivity. In particular, since then $\rk(H_g(1,j)_{(1:j,1:k\})}) = k$, $j \geq k$, the equivalence between item two and three in Theorem~\ref{prop:con_minor} follows from Proposition~\ref{prop:consecutive}. The remaining equivalences are shown in the appendix.

\subsection{Realization of compound systems}
Next, we would like to proof Theorem~\ref{thm:decomp_tpk} and Propositions~\ref{prop:toep_inf_pos} and \ref{prop:relax}. We begin by applying Lemma~\ref{lem:compound_mat} to (\ref{eq:Hankel_mat_real}), which yields \begin{equation*}
	g_{[j]}(t) = \det(H_g(t,j)) = \compound{\Obs{j}(A,c)}{j} \compound{A}{j}^{t-1} \compound{\Con{j}(A,b)}{j}, 
	\end{equation*}
i.e, $G_{[j]}(z)$ admits the state-space realization
\begin{equation}
( \compound{A}{j}, \compound{\Con{j}(A,b)}{j}, \compound{\Obs{j}(A,c)}{j}). \label{eq:ss_compound}
\end{equation}
Note that $g_{[j]} \equiv 0$ for $j > n$, since $\rk(H_g(t,j)) \leq n$, meaning that \emph{Hankel $n$-positivity is equivalent to Hankel total positivity}. A diagonal state-space representation leads then to a characterization of the transfer function $G_{[j]}(z)$ directly in terms of the partial fraction expansion of $G(z)$.
\begin{thm}\label{thm:tpk_poles_coeff}
	Let $G(z) = \sum_{i=1}^{n} \frac{r_i}{z-p_i}$. Then, for $2\leq j \leq n$ it holds that 
\begin{equation}
		G_{[j]}(z) =  \sum_{v \in \mathcal{I}_{n,j}} \frac{\prod_{i=1}^j r_{v_i} \prod_{(i,j) \in \mathcal{I}_{n,2}} (p_{v_i} - p_{v_j})^2}{z-\prod_{i=1}^j p_{v_i}}. \label{eq:compound_tf}
	\end{equation}
Consequently, if $G(z)$ is Hankel $\tp{k}$, then $r_1,\dots,r_k > 0$ and $p_1,\dots,p_k \geq 0$.
\end{thm}
The last claim is a consequence of external positivity of $G_{[j]}$, $1\leq j \leq k$ from Theorem~\ref{prop:con_minor} and Proposition~\ref{prop:dominantpole}. In the appendix, we show how Theorem~\ref{thm:decomp_tpk} follows from Theorem~\ref{thm:tpk_poles_coeff} as well as the following generalization of the third item in Proposition~\ref{prop:relax}.
 \begin{lem}
	\label{lem:diff_Hank}
	Let $G(z)$ be externally positive and $g_d := -\Delta^{(1)} g$. If $G_d(z)$ is Hankel $\tp{k}$, then $G(z)$ is Hankel $\tp{k}$. Vice-versa, if $G(z) = \sum_{i = 1}^n \frac{r_i}{z-p_i}$ with $k_i > 0$ and $p_i \geq 0$, then $G_d(z)$ is Hankel totally positive.  
\end{lem}
Indeed, the equivalence between the first and third item in Proposition~\ref{prop:relax} is a recursive application of Lemma~\ref{lem:diff_Hank} to $g_d$. Further, since $G_{[n]}(z)$ is a first-order system, checking its external positivity is equivalent to $H_g(1,n) = \compound{g}{n}(1) > 0$ and $H_g(2,n) = g_{[n]}(2) \geq 0$. Thus, by Proposition~\ref{prop:toep_inf_pos}, Theorem~\ref{prop:con_minor} and (\ref{eq:psd}), we recover the equivalence between the first and second item in Proposition~\ref{prop:relax}. 

Hence, once we have verified that we can recover Proposition~\ref{prop:toep_inf_pos}, our analysis covers Proposition~\ref{prop:relax}. For systems with simple poles, this is immediate from Theorem~\ref{prop:con_minor} together with a recursive application of Theorem~\ref{thm:decomp_tpk}. To see that Theorem~\ref{thm:decomp_tpk} remains true in case of repeated poles, one should notice that if $$G(z) = \sum_{a=1}^{l} \sum_{b=1}^{m_a} \frac{r_{ba}}{(z-p_a)^b}$$ with $m_1 > 0$, then through our compound system realization, it is easy to show that $G_{[1]}(z)$ has a repeated dominant real pole. However, this contradicts the log-convexity of $g$ on $\mathds{Z}_{>0}$ in Theorem~\ref{prop:con_minor}. Inductively, we can conclude the same for $G_{[j]}$ and $m_j$ with $1\leq j \leq k-1$. In particular, if $p_{k-1} = 0$, then $G_{[j]}$ is of order $k-1$. We summarize this as follows.
\begin{prop}
	Let $G(z) = \sum_{a=1}^{l} \sum_{b=1}^{m_a} \frac{r_{ba}}{(z-p_a)^b}$ be Hankel $\tp{k}$. Then, $m_1 = \dots = m_{k-1} = 1$ and $p_{k-1} > 0$ for $k \leq \sum_{a=1}^l {m_a}$. In particular, if $G(z)$ is Hankel totally positive, then all poles are simple.   
\end{prop}

\section{Toeplitz $k$-positivity}
\label{sec:toep}
The analysis of  Toeplitz $\tp{k}$ systems essentially follows the same steps as in the previous section,  by replacing the Hankel matrix $H_g(t,j)$ with the Toeplitz matrix $T_g(t,j)$. The key observation is to note that if $t-j \geq 0$, then $\det(T_g(t,j)) = \xi(j) \det(H_g(t-j+1,j))$ with 
\begin{equation}
	\xi(j) := \begin{cases}
	1 & j \ \mathrm{mod} \ 4 \leq 1\\ 
	-1 & \text{else}
	\end{cases}.
\end{equation}
This yields the following analogue to Theorem~\ref{prop:con_minor}.
\begin{thm} \label{prop:con_minor_Toep}
Let $G(z)$ and $k \in \mathbb{Z}_{\geq 1}$ be such that the $k-1$-th largest pole is non-zero. Then, we have the equivalences: 
\begin{enumerate}[I.]
	\item $G(z)$ is Toeplitz \ovd{k-1}.
	\item $G(z)$ is Toeplitz $\tp{k}$.
	\item $\xi(j) G_{[j]}$ is externally positive for $1 \leq j \leq k$ and $\det(T_g(t,j)) > 0$ for $t_0\leq t \leq j-1$ and $1 \leq j \leq k-1$, where $t_0 := \min \{t: g(t) \neq 0\}$.
\end{enumerate}  
\end{thm}
\begin{rem}
The  equivalence between the first two items in Theorem~\ref{prop:con_minor_Toep} remains true if the $k-1$-th largest pole is zero. However, in order to be able to apply Proposition~\ref{prop:consecutive}, we implicitly added the assumption that none of the first $k-1$ compound systems is a non-zero FIR system. This is not necessary in the Hankel case, since it is implied by log-convexity of $G_{[j]}$, $1\leq j \leq k-1$ (see~Theorem~\ref{prop:con_minor}).
\end{rem}
Then, Theorem~\ref{thm:tpk_poles_coeff} provides a characterization of the poles and coefficients. 
\begin{cor}\label{cor:tpk_coeff}
	Let $G(z) = \sum_{i=1}^{n} \frac{r_i}{z-p_i}$, $k \in \mathbb{Z}_{>0}$ and $m := \min\{k,n\}$. If $G(z)$ is Toeplitz $\tp{k}$, it holds that $(-1)^{i+1} r_i >0$ for $1\leq i\leq m$ and $p_1,\dots,p_m \geq 0$.
\end{cor}
For $k=2$, this means that while the sum of two first order lags is Hankel totally positive, its difference is Toeplitz totally positive (provided external positivity). Further, we derive the following analogue to Theorem~\ref{thm:decomp_tpk}.
\begin{thm}\label{thm:decomp_toep}
	Let $G(z) =  \sum_{i=1}^{n} \frac{r_i}{z-p_i}$ have distinct poles and $2\leq k \leq n$. If $G(z)$ is Toeplitz $\tp{k}$, then $G(z) = \frac{z}{z-p_1}G_r(z)$, where $G_r(z)$ is Toeplitz $\tp{k-1}$ and all real zeros are smaller than $p_{\min\{k,n\}}$. 
\end{thm}

\section{Examples}
\label{sec:app}
In this section, we will discuss several examples of \ovd{k} systems, which we use to illustrate our theory as well as to point out future directions and relationships to recent related work. 

\subsection{Interconnections of first-order lags}
Consider the two systems
\begin{equation*}
G_a(z) = \frac{r_1}{z-p_1} \; \text{ and } \; G_b(z)= \frac{r_2}{z-p_2} - \frac{r_3}{z-p_3}
\end{equation*}
where it is assumed that $r_1 >0$, $r_2 \geq r_3 > 0$ and $p_2 > p_3 > 0$. The impulse responses
\begin{equation*}
    g_a(t) = r_1p_1^{t-1} \; \text{ and }\; g_b(t) = (r_2p_2^{t-1} - r_3p_3^{t-1}),\; t\geq 1
\end{equation*}
are nonnegative and thus both systems are Toeplitz and Hankel \ovd{0}. In particular, since
\begin{align*}
     (\Hank{g_a}u)(t) = r_1 p_1^t \sum_{\tau=\infty}^{-1} p_1^{-\tau-1} u(\tau) =:  r_1 p_1^t x(0)
\end{align*}
with $x(0)$ a constant, $\vari{\Hank{g_a}u} = 0$ and $\Hank{g}u \geq 0$ for all $u \geq 0$. Thus, proving that 
$G_a(z)$ is Hankel totally \ovd{}. Similarly,
\begin{align*}
  (\Toep{g_a}u)(t) = r_1 p_1^t \sum_{\tau = 0}^{t-1} p_1^{-\tau-1} u(\tau) =: r_1p_1^t x(t)
\end{align*}
and since the damped summation $x$ is only able to follow the sign changes of $u$, it follows that $G_a(z)$ is also Toeplitz totally \ovd{}. The compound systems introduced in Theorems~\ref{prop:con_minor} and \ref{prop:con_minor_Toep} simplify this analysis: it suffices to see that ${G_a}{[j]}(z) = 0$, $j \geq 2$.

Next note that $-{G_b}_{[2]} = \frac{r_2 r_3 (p_2 - p_3)^2}{z - p_2 p_3}$ is externally positive (see~Theorem~\ref{thm:tpk_poles_coeff}), which by Theorems~\ref{prop:con_minor} and \ref{prop:con_minor_Toep} implies that $G_b(z)$ is Toeplitz, but not Hankel, totally \ovd{}. This shows that the difference of Hankel totally positive systems does not remain totally positive (cf.~Lemma~\ref{lem:intercon}). 

Finally, let us consider $G(z) = G_a(z) + G_b(z)$ with $p_1 > p_2$. As the sum of externally positive systems, $G(z)$ is Hankel and Toeplitz \ovd{0}. However, by Corollary~\ref{cor:tpk_coeff} and Proposition~\ref{prop:toep_inf_pos}, $G(z)$ can neither be Toeplitz \ovd{1} nor Hankel totally \ovd{}. To determine when $G(z)$ is Hankel \ovd{1}, we require by Theorem~\ref{prop:con_minor} that 
\small{\begin{equation*}
G_{[2]}(z) = \frac{r_1r_2 (p_1-p_2)^2}{z-p_1p_2}- \frac{r_1r_3 (p_1-p_3)^2}{z-p_1p_3} -  \frac{r_2 r_3 (p_2-p_3)^2}{z-p_2p_3}
\end{equation*}} \normalsize
is externally positive. This is the case if and only if $$r_1 r_2 (p_1-p_2)^2 \geq r_1r_3 (p_1-p_3)^2 + r_2 r_3 (p_2-p_3)^2.$$ 
An illustration of these properties for the Hankel and Toeplitz operators of $G(z)$ are given in Figure~\ref{fig:ex1}.

\begin{figure}
    \centering
    \begin{tikzpicture}
    	\begin{groupplot}[group style={group name=my plots, group size=1 by 3,vertical sep = 3 cm}]
			\nextgroupplot[height=3cm,
				width=15cm,
			axis y line = left,
			axis x line = center,
			ylabel={$(\Hank{g}u)(t)$}, 
			xlabel={$t$}, 
			xlabel style={right},
			xticklabel style={above,yshift=0.5 ex},
			xmax = 7.9,
			]			
			\addplot+[ycomb,black,thick]
				table[x index=0,y index=1] 
				{output_hankel_dim.txt};
				
			\nextgroupplot[height=3cm,
			width=15cm,
			axis y line = left,
			axis x line = center,
			ylabel={$(\Toep{g}u)(t)$}, 
			xlabel={$t$}, 
			xlabel style={right},
			xticklabel style={above,yshift=1.7ex},
			xmax = 7.9,
			ymax = .9,
			]			
			\addplot+[ycomb,black,thick]
				table[x index=0,y index=1] 
				{output_toep.txt};
				
			\nextgroupplot[height=3cm,
			width=15cm,
			axis y line = left,
			axis x line = center,
			ylabel={$(\Hank{g}u)(t)$}, 
			xlabel={$t$}, 
			xlabel style={right},
			xticklabel style={above,yshift=0.5ex},
	        xmax = 7.9,
			]			
			\addplot+[ycomb,black,thick]
				table[x index=0,y index=1] 
				{output_hankel.txt};	
		\end{groupplot}
			\node[text width=0.97 \textwidth ,align=center,anchor=north] at ([yshift=-3mm,xshift = -1 mm]my plots c1r1.south) {\subcaption{\setlength{\tabcolsep}{1.5 pt}Output of the Hankel operator, where $u$ is zero except for $u(-1:-2) = (1,-10)$.
			\label{subfig:hankel_dim}}};%
			\node[text width=0.97\textwidth ,align=center,anchor=north] at ([yshift=-3mm,xshift = -1 mm]my plots c1r2.south) {\subcaption{\setlength{\tabcolsep}{1.5 pt}Output of the Toeplitz operator, where $u$ is zero except for $u(0:1) = (10,-8.5)$ and $(\Toep{g}u)(1) = 13$.
			\label{subfig:toep}}};%
			\node[text width=0.97 \textwidth ,align=center,anchor=north] at ([yshift=-3mm,xshift = -1 mm]my plots c1r3.south) {\subcaption{\setlength{\tabcolsep}{1.5 pt} Output of the Hankel operator, where $u$ is zero except for $u(-1:-3) = (10.9,-21.5, 9.7)$.  
			\label{subfig:hankel}}};%
    \end{tikzpicture}
    \caption{Hankel and Toeplitz operator responses for the Hankel \ovd{1} system $G(z) = G_a(z) + G_b(z)$ with $r_1 = p_1 = 0.9$, $r_2 = p_2 = 0.5$ and $r_3 = p_3 = 0.1$: (\ref{subfig:hankel_dim}) $\vari{u} = 1$ is diminished to $\vari{\Hank{g}u} = 0$; since $\vari{\Hank{g}u} \neq \vari{u}$, it is not required that $\sign((\Hank{g}u)(0))= \sign(u(-1))$. (\ref{subfig:toep}) $\vari{u} = 1$ is increased to $\vari{\Toep{g}u} = 2$; thus, $\Toep{g}$ cannot be \ovd{1}. (\ref{subfig:hankel}) $\vari{u} = \vari{\Hank{g}u} = 2$, but $\Hank{g}$ cannot be \ovd{2}, because $\sign((\Hank{g}u)(0)) \neq \sign(u(-1))$. \label{fig:ex1}}
\end{figure}

\subsection{Static Non-Linearities}
Even though the \ovd{k} property has been traditionally studied for linear mappings \cite{karlin1968total}, as an input-output property, it can also be defined for non-linear systems. As a first step towards a non-linear systems theory, the case of an LTI system with static output non-linearity is discussed.

We begin by noticing that a static non-linearity $\sigma(y)$, which is nonnegative for $y \geq 0$ and nonpositive for $y < 0$, is variation preserving, i.e., $\vari{\Hank{g}u} = \vari{\sigma(\Hank{g}u)}$. Thus, $\sigma(\Hank{g}u)$ is \ovd{k} if and only if $G(z)$ is Hankel \ovd{k}. Non-linearties with this property are commonly found, e.g., in a (dead-zone) relay or a saturation. 

For nonnegative $\sigma$, such as the sigmoid activation fiction, the \ovd{k} property may seem less informative, since $\vari{\sigma(\Hank{g}u)} = 0$ for all $G(z)$ and $u$. However, as the sigmoid function is monotonically non-decreasing, it follows from the chain rule that it preserves the number of local extrema. Then, since $(\Hank{g}\Delta^{(1)}u)(t) = \Delta^{(1)}(\Hank{g}u)(t)$, $\sigma(\Hank{g}u)$ is order-preserving local extrema diminishing if and only if $G(z)$ is Hankel \ovd{k}. 
Single input neurons or logistic regression units as modelled in Figure~\ref{fig:percep} are, therefore, local extrema diminishing. %
As such models typically come with multiple in- and outputs, the study of variation diminishing multi-input-multi-output systems is as an important future task.  

Analogues considerations also apply to the Toeplitz operator. 
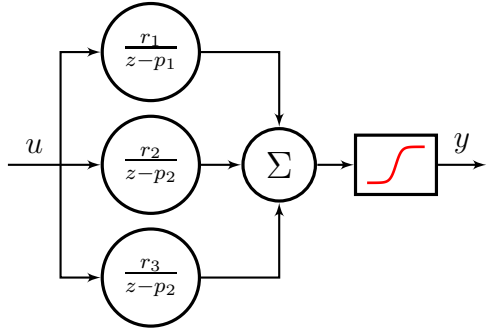
\begin{figure}
    \centering
\tikzstyle{int}=[draw,minimum width=1cm, minimum height=1cm, very thick, align = center]
\tikzstyle{neu}=[draw, very thick, align = center,circle]
\begin{tikzpicture}[>=latex',circle dotted/.style={dash pattern=on .05mm off 1.2mm,
			line cap=round}]
	\node (a) [neu] {$\frac{r_2}{z-p_2}$}; 
	\node (a2) [neu, above of=a, node distance=1.5cm] {$\frac{r_1}{z-p_1}$};
	
	\node (a3) [neu, below of=a, node distance=1.5cm] {$\frac{r_3}{z-p_2}$};

	\node (b) [left of=a,node distance=1.2cm, coordinate]{};
	\node (u) [left of =b, node distance=.7 cm, coordinate]{};
	
	\path[thick] (u) edge node[above]{$u$} node(u_inc)[above]{} (b);
	\draw[->,thick] (b) |- (a2);
	\draw[->,thick] (b) |- (a3);
	\draw[->,thick] (b) |- (a);

	\node [coordinate] (end) [right of=a, node distance=1.2cm]{};
	\path[->,thick] (a) edge node[above]{} (end);
	
	\node (sum) [neu, at = (end), minimum height = .5 cm, anchor = {west}, scale = 1.2]{$\Sigma$};
	
	\node [coordinate] (end2) [right of=end, node distance=1.5cm]{};

    \draw[->,thick] (a2) -| (sum);
    
    \draw[->,thick] (a3) -| (sum);
	
	\node (sig) [int, at = (end2), minimum height = .5 cm, anchor = {west}] {
	\begin{tikzpicture}
	\begin{axis}[ticks = none,width = 2.5 cm, axis lines = none]
	\addplot[line width = 1 pt, color = FigColor2] file{sigmoid.txt};
	\end{axis}
	\end{tikzpicture}};
	
	\node [coordinate] (mid1) [right of=sig, node distance = 1.2 cm]{};
	\path[->,thick] (sig) edge node[above]{$y$} (mid1);
\draw[->,thick] (sum) -- (sig);

	\end{tikzpicture}
   
    \caption{A single-input perceptron with $r_i,p_i \geq 0$, $i =1,2,3$, is totally Hankel \ovd{} from $\Delta^{(1)}u$ to $\Delta^{(1)}y$, i.e., it diminishes local extrema from past inputs $u$ to future outputs $y$.\label{fig:percep}}
\end{figure}

\subsection{Heavy-ball method}
The so-called heavy-ball method for (convex) optimization was designed by Polyak to damp the possibly high number of local extrema in the iterates of the well-known gradient descent approach \cite{polyak19641methods,ghadimi2015global}. For a convex function $f: \mathbb{R} \to \mathbb{R}$, the iterates $x(k)$ of this method are given by the closed-loop system
\begin{align*}
x({k+1}) &= x(k) + \alpha u(k) + \beta (x(k) - x({k-1}))\\
y(k) &= \left . \frac{d}{dx}f(x) \right |_{x = x_k} \\
u(k) &= -y
\end{align*}
for constant $\alpha,\beta > 0$. Since the linear system from $u$ to $x$ reads $G(z) = \frac{\alpha z}{(z-1) (z - \beta)}$, it is Toeplitz totally \ovd{} by Proposition~\ref{prop:toep_inf_pos}. Further, as the derivatives of convex functions are non-decreasing \cite{Boyd2004}, it follows as above that the entire open loop system from $u$ to $y$ is local extrema diminishing. 

To also analysis the closed-loop system, let us assume that $f$ is quadratic. Then, $y(k) = ax - b$ with $a > 0$, i.e., $y$ is the response of a linear closed-loop system to a step with height $b$. It is easy to show that the closed-loop system is Toeplitz totally \ovd{} if and only if $\beta \geq (\sqrt{a \alpha}+1)^2$. This observation verifies the intuition that the required amount of momentum (damping) for variation diminishment scales with the the gain of the gradient, which in the optimization literature is associated with the Lipschitz constant of the gradient \cite{ghadimi2015global}. Then, choosing $x(0) = 0$ results in iterates without local extrema. In particular, local extrema from noisy gradients (which may be modelled as time-varying $b$) can be diminished independently of the noise distribution. Since suggested choices for $\alpha$ and $\beta$ \cite{ghadimi2015global,polyak19641methods} do not necessarily result in closed-loop total \ovd{}, an investigation of the effects of Toeplitz $\tp{k}$ designs to convergence rates and robustness seems highly desirable. 

This application illustrates the relevance of variation diminishing theory in optimization and motivates the extension of our theory to non-linear negative feedback interconnections. Note that in case of multi-dimensional iterates, the algorithm as well as our analysis applies component-wise. Further, since $G(z)$ admits a minimal realization $(A,b,c)$ with totally positive $A$, it becomes evident that the  analysis of external k-positivity in this paper is related to  the study of $k$-positive dynamical systems \cite{weiss2019generalization,margaliot2018revisiting}.

\section{Conclusion}
\label{sec:conc}
We have studied the variation diminishing property of the Toeplitz and the Hankel operator of finite-dimensional causal LTI systems. Each class defines different refinements of the system property of  external positivity. While the theory of totally positive operators is classical in mathematics \cite{karlin1968total}, the state-space realization of finite-dimensional LTI systems sheds new light onto the properties of $k$-totally positive systems. In particular, we have provided a bridge between external positivity and total positivity. A key insight of the present paper is that the $k$-positivity of a LTI system can be studied via the classical positivity property of associated compound systems.

The present work has focused on external open-loop k-positivity. Future work should investigate the relationship to the internal property as studied in the recent work \cite{margaliot2018revisiting,weiss2019generalization} as well as feedback interconnections. Implications of $k$`-positivity for model reduction are studied in \cite{grussler2020balanced}.

\section*{Acknowledgment}
The research leading to these results was completed while the first author was a postdoctoral research associate at the University of Cambridge. The research has received funding from the European Research Council under the Advanced ERC Grant Agreement Switchlet n.670645.

\appendix
\section*{Appendix}
\label{sec:app}

\section{Proof of Proposition~\ref{prop:SR}}
\label{proof:prop:SR}
By \cite[Theorem~5.2.4]{karlin1968total}, it follows that $X$ being $\tp{k}$ implies that $X$ is \ovd{k-1}. To show the converse, first note that if $X$ is \ovd{k-1}, then also any submatrix $X_{(1:n),J} \in  \mathds{R}^{n \times k}$ is \ovd{k-1}, because one can choose $u$ such that $u_i = 0$ for all $i \in (1:m) \setminus J$. If we can show that $X_{(1:n),J} \in  \mathds{R}^{n \times k}$ is then $\tp{k}$ for all such $J$, our claim follows. To this end note that $X_{(1:n),J}$ is $\tp{k}$ provided that $\rk(X_{(1:n),J}) = k$ by \cite[Theorem~5.1.5]{karlin1968total}. If $\rk(X_{(1:n),J}) = r < k$, then all $j$-minors of $X_{(1:n),J}$ with $j > r$ are zero and it suffices to verify that $X_{(1:n),\hat{J}} \in \mathds{R}^{n \times r}$ is $\tp{r}$ for all $\hat{J} \subset J$. As before, if $\rk(X_{(1:n),\hat{J}}) = r$, then $X_{(1:n),\hat{J}}$ is $\tp{r}$ and otherwise it suffices to consider all smaller subsets of columns.

\section{Proof of Proposition~\ref{prop:op_seq}}
\label{proof:op_seq}
The proof uses the following simple lemma.
\begin{lem}\label{lem:seq}
	Let $\{x^k\}_{k \in \mathds{N}}$ with $x^k \in \mathbb{R}^n$ and $x^\opts := \lim_{k \to \infty} x^k \in \mathbb{R}^n$. Then, $\lim_{k \to \infty} \vari{x^k} \geq \vari{x^\opts}$.
\end{lem}
\begin{proof}
	By the convergence of $x^k$, there exists $N \in \mathbb{N}$ such that $\forall i: \ x^N_ix^\opts_i > 0$ if $x^\opts_i \neq 0$. Thus, $\vari{x^N} \geq \vari{x^\opts}$.
\end{proof}
	First note that for all $u$ with $\Hank{g}u \in \linf$ and $\vari{\Hank{g}u} < \infty$, there exists an $N_u \in \mathbb{N}$ such that 
	\begin{align}
	\vari{(\Hank{g}u)(0:N_u-1)} = \vari{\Hank{g}u}
	\end{align}
     As $\lim_{j \to \infty} (\Hank{g}^ju)(0:N_u-1) = (\Hank{g}u)(0:N_u-1)$, it holds by Lemma~\ref{lem:seq} that
	\begin{equation}
	\lim_{j \to \infty} \vari{\Hank{g}^ju} \geq \vari{\Hank{g}u},
	\end{equation} 
	which under the assumption of the second item implies the first. Conversely, if the first item holds, then 
	\begin{align*}
	k-1 \geq \vari{u} &\geq \vari{u\mathds{1}_{[-j,-1]}} \geq \vari{\Hank{g}(u\mathds{1}_{[-j,-1]})} \\
	&\geq \vari{(\Hank{g} (u \mathds{1}_{[-j,-1]}))(0:j)} = \vari{\Hank{g}^j u}.
	\end{align*}

\section{Proof of Theorem~\ref{prop:con_minor}}
The equivalences between the first three items have been discussed in the main text. Since item five implies item three, we are left with showing that item three implies items four and five and item four implies item three. We first show these for the strictly Hankel $\tp{k}$ case. 

 Item three is then equivalent to $H_g(t,k)$ being strictly totally positive for all $t \geq 1$ by Proposition~\ref{prop:consecutive}, which by (\ref{eq:psd}) is equivalent to $H_g(t,j) \succ 0$ for all $1\leq j \leq k$ and $t\geq 1$. In particular, this implies item four. Furthermore, by Lemma~\ref{lem:compound_mat}, $\compound{H_g(t,k)}{j} \succ 0$ for all $1 \leq j \leq k$ and $t \geq 1$, which through deletions of columns and rows implies that $H_{\compound{g}{j}}(t,i) \succ 0$ for all $1\leq i \leq k-j+1$. By the equivalence above, this implies that $G_{[j]}$ is strictly Hankel $\tp{k-j+1}$.

To see that item four implies item three, we next show the following: if $g_{[k]} > 0$ and $g_{[j]}(t) > 0$ for all $1 \leq t \leq 2^{l+k-j}$ and $1\leq j \leq k-1$ with $l \in \mathds{Z}_{\geq 0}$, then also $g_{[j]}(t) > 0$ for $ 1\leq t \leq 2^{k-j+l+1}$. Induction over $l$, where item four corresponds to the case $l=0$ by (\ref{eq:psd}), then implies item three. To show our claim, first note that by (\ref{eq:Sylvester_id})
\begin{equation}
    g_{[j-2]}(t+2) g_{[j]}(t) = g_{[j-1]}(t) g_{[j-1]}(t+2) - g_{[j-1]}^2(t+1). \label{eq:sylvester_hank}
\end{equation}
with $g_{[0]} \equiv 1$. Then by assumption, $g_{[k-1]}(t+2) > \frac{g_{[k-1]}^2(t+1)}{g_{[k-1]}(t)}$ and $\compound{g}{k-1}(t) > 0$ for $1 \leq t \leq 2^{l+1}$, which requires that $g_{[k-1]}(2^{l+2}-1), g_{[k-1]}(2^{l+2}) > 0$. The cases of $j < k$ follow analogously by induction. 

Finally by Lemma~\ref{lem:intercon}, our equivalences also follow for non-strict Hankel $\tp{k}$ systems by using non-strict inequalities. Nonetheless, $H_g(1,k-1) \succ 0$ remains strict, because by item five (or (\ref{eq:sylvester_hank})) $g_{[j]}$ is log-convex on $\mathds{Z}_{>0}$ for $1\leq j \leq k-1$ and thus requires that $g_{[j]}(1) >0$.

\section{Proof of Theorem~\ref{thm:tpk_poles_coeff}}
\label{proof:tpk_poles_coeff}
The proof of Theorem~\ref{thm:tpk_poles_coeff}  uses the following simple lemma, which follows by applying \cite[p.~37]{horn2012matrix} to 
\begin{equation*}
\Con{n}(A,b)) = \diag(b) \begin{pmatrix} 1 & p_1 & \dots & p_1^{n-1}\\
1 & p_2 & \dots & p_2^{n-1}\\
\vdots & \vdots & \vdots \\
1 & p_n & \dots & p_n^{n-1}
\end{pmatrix}.
\end{equation*}
\begin{lem}\label{lem:diagonal}
	For $b \in \mathds{C}^n$, $A = \diag(p_1,\dots,p_n) \in \mathds{C}^{n \times n}$
	\begin{equation}
	\det(\Con{n}(A,b)) = (-1)^n \prod_{i=1}^n b_i \prod_{(i,j) \in \mathcal{I}_{n,2}} (p_j - p_i). \label{eq:diag_ctrl}
	\end{equation}
\end{lem}

Let $G(z)$ be realized by the matrices $A = \diag(p_1,\dots,p_n)$ and $c = b^\transp = \begin{pmatrix}
\sqrt{r_1} & \dots & \sqrt{r_n}
\end{pmatrix}$ 
and $(\bar{A},\bar{b},\bar{c})$ denote the corresponding $j$-th compound system realization in (\ref{eq:ss_compound}). Then $\bar{A}$ is diagonal and $\bar{b} = \bar{c}^\transp$, where the $l$-th diagonal entry in $\bar{A}$ reads $\prod_{i=1}^j p_{v_i}$, $v$ being the $l$-th element in $\mathcal{I}_{n,j}$. Further, ${\bar{b}}_l = \det(\Con{j}({A}_{v,v},b_v)) = \prod_{i=1}^j \sqrt{r}_{v_i} \prod_{(i,j) \in \mathcal{I}_{n,2}} (p_{v_j} - p_{v_i})$ by Lemma~\ref{lem:diagonal}. Thus, the claimed form of the transfer function follows. The last claim then follows from applying Proposition~\ref{prop:dominantpole} to the external positivity of $\compound{G}{j}(z)$ (see~Theorem~\ref{prop:con_minor}).  
	
\section{Proof of Lemma~\ref{lem:diff_Hank}}
\begin{proof} 
$\Rightarrow$: By definition of Hankel $k$-positivity and (\ref{eq:psd})), it holds that $H_{-\Delta^{(1)} g}(t,j) = H_{g}(t,j) - H_{g}(t+1,j) \succeq 0$ for all $t \geq 1$ and $1 \leq j \leq k$. Therefore, $H_{g}(t,j) - \lim_{t \to \infty} H_g(t,j)= \sum_{l=t}^{\infty} H_{-\Delta^{(1)} g}(l,j) \succeq 0.$
	Since $G_1(z) = \frac{r_1}{z-p_1}$ is totally positive and \linebreak $\lim_{t \to \infty} H_g(t,j) = \lim_{t \to \infty} H_{g_1}(t,j)$, we get that $H_{g}(t,j) \succeq 0$, which by Theorem~\ref{prop:con_minor} proves the claim.
	
$\Leftarrow$: It is easy to verify that $H_{-\Delta^{(1)} g_i}(t,j) \succeq 0$ for all $j$ and each $G_i(z) = \frac{k_i}{z-p_i}$. Therefore, $G_d(z)$ is Hankel totally positive by Theorem~\ref{prop:con_minor} and Lemma~\ref{lem:intercon}.
\end{proof}
\section{Proof of Theorem~\ref{thm:decomp_tpk}}

It suffices to show the result under the assumption that $p_1 = 1$ and $p_n \neq 0$. To see this note that the case of $p_1 = 0$ is trivial and multiplying $g(t)$ with $p_1^t$, $p_1 > 0$, i.e., a Hankel totally positive first order lag allows us to recover the case $0 < p_1 < 1$ by Lemma~\ref{lem:intercon}. Finally, if $p_n =0$, then $r_n$ only affects $H_g(1,j)$ and otherwise the system can be treated as of order $n-1$. Therefore, let $g(t) = (r_1 + g_r(t))s(t-1)$. Since $\Delta^{(1)}g = \Delta^{(1)}g_r$, the claim is proven by Lemma~\ref{lem:diff_Hank} once we can show that $\Hank{-\Delta^{(1)}g}$ is $\tp{k-1}$.

By Proposition~\ref{prop:consecutive}, it is suffices to show that $\det(H_{-\Delta^{(1)}g}(t,j)) > 0$ for all $1 \leq j \leq k-1$.  
We begin by noticing that $\det(H_{g}(t,j)) > 0$ for $t\geq 1$ and $j \leq n$, because $p_n \neq 0$. In particular, successive row subtractions yield that	\begin{equation}
	0 < \det(H_{g}(t,j)) = \det(\tilde{H}_g(t,j)) \label{eq:diff_comp_equiv}
	\end{equation}
	with
	\small
	\begin{equation*}
	\tilde{H}_g(t,j) := \begin{pmatrix}
	-\Delta^{(1)}g(t)& \dots & -\Delta^{(1)}g(t+j-1)\\
	\vdots &      & \vdots \\
	-\Delta^{(1)}g(t+j-2)   & \dots & -\Delta^{(1)}g(t+2j-2)\\
	g(t+j-1)  & \dots & g(t+2j-2)
	\end{pmatrix}.
	\end{equation*}\normalsize
    for all $1 \leq j \leq k$ and $t \geq 1$. 
	Thus, 
	\begin{multline*}
	\label{eq:var1_hankel}
	\det(H_g(t+2,j-1))  \det(H_{-\Delta^{(1)}g}(t,j-1)) > \\
	\det(H_g(t+1,j-1))\det(H_{-\Delta^{(1)}g}(t+1,j-1)).
	\end{multline*}
	By applying (\ref{eq:Sylvester_id}) to $\det(\tilde{H}_g(t,j))$ and using the substitution in (\ref{eq:diff_comp_equiv}). In other words, \linebreak $\det(H_{-\Delta^{(1)}g}(\cdot,j-1))$ can switch sign at most once from positive to negative. However, by letting $G_i(z) := \frac{r_i}{z-p_i}$ for $i \geq 2$, it follows for the dominant dynamics $\sum_{i=1}^j G_i(z)$ of $G_r(z)$ by Theorem~\ref{thm:tpk_poles_coeff} and Lemma~ \ref{lem:diff_Hank} that  
	$$H_{-\Delta^{(1)}\sum_{i=2}^{j} g_i }(t,j-1) = \sum_{i=2}^j H_{-\Delta^{(1)} g_i}(t,j-1) \succ 0,$$  
	which implies that $\det(H_{-\Delta^{(1)}g}(t,j-1)) > 0$ for sufficiently large $t \geq 1$. Hence, \linebreak $\det(H_{-\Delta^{(1)}g}(t,j-1)) > 0$ for all $t \geq 1$.

\section{Proof of Theorem~\ref{thm:decomp_toep}}
	The proof is similar to the Hankel case. Using 
	Theorem~\ref{prop:con_minor_Toep}, our goal is to show that if $p_1 = 1$ then $\det(T_{g_r}(t,j)) > 0$ for all $1\leq j \leq k-1$, where $G(z) = \frac{z}{z-1} G_r(z)$, i.e., $g(i) := \sum_{\tau=0}^{i} g_r(\tau)$. Noticing that $g(t) p_1^{t-1}s(t-1)$ defines a Toeplitz $\tp{k}$ systems for $1 > p_1 > 0$ , recovers the asymptotically stable case. 
	
	Note  that by (\ref{eq:Sylvester_id}), it holds that $g_{[j]}$ is log-concave for all $1 \leq j \leq k-1$. Then, by the assumption that $G(z)$ has at most a simple pole in zero, it follows from Theorem~\ref{thm:tpk_poles_coeff} that $g_{[j]}(t)$, $t \geq 0$, can only be zero on some initial finite interval. Since the system remains Toeplitz \ovd{k-1} under time shifts, we can assume by Theorem~\ref{prop:con_minor_Toep} that $\det(T_{g}(t,j)) > 0$ for all $t \geq 1$ and $1 \leq j \leq k-1$. For $1 \leq j \leq k$ we have then
	\begin{equation}
	\det(T_g(t,j)) \geq 0 \ \Leftrightarrow \ \det(\tilde{T}(t,j)) \geq 0 
	\end{equation}
	with
	\begin{equation}
	\tilde{T}{(t,j)} := \begin{pmatrix}
	g(t)  & \dots & g(t-j+1)\\
	g_r(t+1)& \dots & g_r(t-j)\\
	\vdots &   & \vdots \\
	g_r(t+j-1)  & \dots & g_r(t)\\
	\end{pmatrix}. \label{eq:trans_K}
	\end{equation}
	Using (\ref{eq:Sylvester_id}) as well as the facts that $\det(T_g(t,j-1)) = \det(\tilde{T}_{g}(t,j-1))$ yields
	\begin{multline}
	\label{eq:var1_toep}
	    \det(T_g(t,j-1))  \det(T_{g_r}(t,j-1)) \geq \\
	\det(T_g(t-1,j-1))\det({T}_{g_r}(t+1,j-1)).
	\end{multline}
By way of contradiction, let $j$ be the smallest integer such that that there exists $t^\ast > 0$ with $\det(T_{g_r}(t^\ast,j-1)) \leq 0$. Then, $\det(T_{g_r}(t,j-1)) \leq 0$ for all $t \geq t^\ast$ by  (\ref{eq:var1_toep}), i.e., the dominant dynamics of $-{G_r}_{[j-1]}(z)$ have to be externally positive. By Theorem~\ref{thm:tpk_poles_coeff} and Corollary~\ref{cor:tpk_coeff}, this implies that $$G_r(z) = \sum_{i=2}^{n} \frac{\bar{r}_i}{z-p_i} + \frac{\bar{r}_0}{z}$$ 
with $(-1)^{j} \bar{r}_{j-1} \leq 0$. 
However, by Corollary~\ref{cor:tpk_coeff} and partial fraction decomposition of $G(z)$, it is easy to verify that $(-1)^{j} \bar{r}_{j-1} > 0$, which provides the contradiction.

\bibliographystyle{plain}

\bibliography{refkpos,refopt,refpos,science}

\begin{thebibliography}{10}

\bibitem{Aissen1952generating}
Michael Aissen, I.~J. Schoenberg, and A.~M. Whitney.
\newblock On the generating functions of totally positive sequences {I}.
\newblock {\em Journal d'Analyse Math{\'e}matique}, 2(1):93--103, Dec 1952.

\bibitem{blondel2002presence}
Vincent~D. Blondel and Natacha Portier.
\newblock The presence of a zero in an integer linear recurrent sequence is
  {NP}-hard to decide.
\newblock {\em Linear Algebra and its Applications}, 351:91 -- 98, 2002.

\bibitem{Boyd2004}
S.P. Boyd and L.~Vandenberghe.
\newblock {\em Convex Optimization}.
\newblock Berichte {\"u}ber verteilte messysteme. Cambridge University Press,
  2004.

\bibitem{boyd2004convex}
Stephen Boyd and Lieven Vandenberghe.
\newblock {\em Convex Optimization}.
\newblock Cambridge University Press, 2004.

\bibitem{drummond2019external}
R.~{Drummond}, M.~C. {Turner}, and S.~R. {Duncan}.
\newblock External positivity of linear systems by weak majorisation.
\newblock In {\em 2019 American Control Conference (ACC)}, pages 5191--5196,
  2019.

\bibitem{elkhoury1993discrete}
M.~El-Khoury, O.D. Crisalle, and R.~Longchamp.
\newblock Discrete transfer-function zeros and step-response extrema.
\newblock {\em IFAC Proceedings Volumes}, 26(2, Part 2):537 -- 542, 1993.
\newblock 12th Triennal Wold Congress of the International Federation of
  Automatic control. Volume 2 Robust Control, Design and Software, Sydney,
  Australia, 18-23 July.

\bibitem{elkhoury1993influence}
Mario El-Khoury, Oscar~D. Crisalle, and Roland Longchamp.
\newblock Influence of zero locations on the number of step-response extrema.
\newblock {\em Automatica}, 29(6):1571 -- 1574, 1993.

\bibitem{fallat2017total}
Shaun Fallat, Charles~R. Johnson, and Alan~D. Sokal.
\newblock Total positivity of sums, {H}adamard products and {H}adamard powers:
  Results and counterexamples.
\newblock {\em Linear Algebra and its Applications}, 520:242 -- 259, 2017.

\bibitem{Farina2000}
L.~Farina and S.~Rinaldi.
\newblock {\em Positive linear systems: theory and applications}.
\newblock Pure and applied mathematics (John Wiley \& Sons). Wiley, 2000.

\bibitem{farina2011positive}
Lorenzo Farina and Sergio Rinaldi.
\newblock {\em Positive Linear Systems: Theory and Applications}.
\newblock John Wiley \& Sons, 2011.

\bibitem{fekete1912uber}
M.~Fekete and G.~P{\'o}lya.
\newblock {\"U}ber ein problem von {L}aguerre.
\newblock {\em Rendiconti del Circolo Matematico di Palermo (1884-1940)},
  34(1):89--120, Dec 1912.

\bibitem{fiedler2008special}
Miroslav Fiedler.
\newblock {\em Special Matrices and Their Applications in Numerical
  Mathematics}.
\newblock Courier Corporation, 2008.

\bibitem{gantmacher1950oszillationsmatrizen}
FR~Gantmacher and MG~Krein.
\newblock Oszillationsmatrizen, oszillationskerne und kleine schwingungen
  mechanischer systeme, akademie-verlag, berlin, 1960.
\newblock {\em Russian original edition: Moscow-Leningrad}, 1950.

\bibitem{ghadimi2015global}
E.~{Ghadimi}, H.~R. {Feyzmahdavian}, and M.~{Johansson}.
\newblock Global convergence of the heavy-ball method for convex optimization.
\newblock In {\em 2015 European Control Conference (ECC)}, pages 310--315,
  2015.

\bibitem{grussler2018strongly}
C.~{Grussler} and R.~{Sepulchre}.
\newblock Strongly unimodal systems.
\newblock In {\em 2019 18th European Control Conference (ECC)}, pages
  3273--3278, 2019.

\bibitem{grussler2020balanced}
Christian Grussler, Tobias Damm, and Rodolphe Sepulchre.
\newblock Balanced truncation of $k$-positive systems.
\newblock arXiv:2006.13333, 2020.

\bibitem{grussler2019tractable}
Christian Grussler and Anders Rantzer.
\newblock On second-order cone positive systems.
\newblock arXiv:1906.06139, 2019.

\bibitem{horn2012matrix}
Roger~A. Horn and Charles~R. Johnson.
\newblock {\em Matrix Analysis}.
\newblock Cambridge University Press, 2 edition, 2012.

\bibitem{ibragimov1956composition}
I.~Ibragimov.
\newblock On the composition of unimodal distributions.
\newblock {\em Theory of Probability \& Its Applications}, 1(2):255--260, 1956.

\bibitem{karlin1968total}
Samuel Karlin.
\newblock {\em Total Positivity}, volume~1.
\newblock Stanford University Press, 1968.

\bibitem{lindeberg1990scale}
T.~{Lindeberg}.
\newblock Scale-space for discrete signals.
\newblock {\em IEEE Transactions on Pattern Analysis and Machine Intelligence},
  12(3):234--254, 1990.

\bibitem{margaliot2018revisiting}
Michael Margaliot and Eduardo~D. Sontag.
\newblock Revisiting totally positive differential systems: A tutorial and new
  results.
\newblock {\em Automatica}, 101:1 -- 14, 2019.

\bibitem{ohta1984reachability}
Yoshito Ohta, Hajime Maeda, and Shinzo Kodama.
\newblock Reachability, observability, and realizability of continuous-time
  positive systems.
\newblock {\em SIAM Journal on Control and Optimization}, 22(2):171--180, 1984.

\bibitem{pates2019optimal}
Richard Pates, Carolina Bergeling, and Anders Rantzer.
\newblock On the optimal control of relaxation systems.
\newblock arXiv:1909.07219, 2019.

\bibitem{polyak19641methods}
B.T. Polyak.
\newblock Some methods of speeding up the convergence of iteration methods.
\newblock {\em USSR Computational Mathematics and Mathematical Physics}, 4(5):1
  -- 17, 1964.

\bibitem{rantzer2015scalable}
Anders Rantzer.
\newblock Scalable control of positive systems.
\newblock {\em European Journal of Control}, 24:72 -- 80, 2015.

\bibitem{Samworth2017}
Richard Samworth.
\newblock Recent progress in log-concave density estimation.
\newblock {\em Statistical Science}, 33(4):493--509, 2018.

\bibitem{Schoenberg1930vari}
I.~J. Schoenberg.
\newblock {\"U}ber variationsvermindernde lineare transformationen.
\newblock {\em Mathematische Zeitschrift}, 32(1):321--328, Dec 1930.

\bibitem{Schoenberg1951polya}
I.~J. Schoenberg.
\newblock On {P}{\'o}lya frequency functions.
\newblock {\em Journal d'Analyse Math{\'e}matique}, 1(1):331--374, Dec 1951.

\bibitem{son1996robust}
N.~K. Son and D.~Hinrichsen.
\newblock Robust stability of positive continuous time systems.
\newblock {\em Numerical Functional Analysis and Optimization},
  17(5-6):649--659, 1996.

\bibitem{tanaka2011bounded}
T.~Tanaka and C.~Langbort.
\newblock The bounded real lemma for internally positive systems and
  {H}-infinity structured static state feedback.
\newblock {\em IEEE Transactions on Automatic Control}, 56(9):2218--2223, 2011.

\bibitem{weiss2019generalization}
E.~{Weiss} and M.~{Margaliot}.
\newblock A generalization of linear positive systems.
\newblock In {\em 2019 27th Mediterranean Conference on Control and Automation
  (MED)}, pages 340--345, 2019.

\bibitem{willems1976realization}
Jan~C. Willems.
\newblock Realization of systems with internal passivity and symmetry
  constraints.
\newblock {\em Journal of the Franklin Institute}, 301(6):605 -- 621, 1976.

\bibitem{yuccer2002discrete}
{\"U}mit Y{\"u}ceer.
\newblock Discrete convexity: convexity for functions defined on discrete
  spaces.
\newblock {\em Discrete Applied Mathematics}, 119(3):297 -- 304, 2002.

\end{thebibliography}

\end{document}